\documentclass[12pt]{article}
\usepackage{amsmath,amsthm,amsfonts,amssymb,amscd,cite}
\usepackage{lineno}

\usepackage{url,floatflt}
\usepackage{helvet,times}
\usepackage{float}
\usepackage[bf,hypcap]{caption}
\usepackage{xcolor}
\usepackage{wrapfig}
\usepackage{graphicx}
\usepackage{setspace}
\usepackage{subfig}
\usepackage{enumerate}
\usepackage{bookmark}
\usepackage{placeins}
\usepackage{multirow}
\usepackage{multicol}
\usepackage[curve,frame,line,arrow,matrix]{xy}
\usepackage[latin5]{inputenc}

\newcommand{\abs}[1]{\left\vert#1\right\vert}
 \newtheorem{theorem}{Theorem}
 \newtheorem{lemma}{Lemma}

\begin{document}

\title{Applications of a theorem by Ky Fan in the theory of weighted Laplacian graph energy}
\author{Reza Sharafdini$^\ast$, Alireza Ataei,  Habibeh Panahbar\\
\small{Department of Mathematics, Persian Gulf University, Bushehr 75169-13817, Iran}}
\date{}
%\titlefigurecaption{\hspace{3cm}{\hspace{5.8cm}\Large \bf \rm New Trends in Mathematical Sciences}}
%\institute {Department of Mathematics, Persian Gulf University, Bushehr 75169-13817, Iran}
%\titlerunning{APPLICATIONS OF A THEOREM BY KY FAN IN THE THEORY OF WEIGHTED LAPLACIAN GRAPH ENERGY}
%\authorrunning{Reza Sharafdini, Alireza Ataei,  Habibeh Panahbar}
%\mail{sharafdini@pgu.ac.ir}
%\received{xx}
%\revised{xx}
%\accepted{xx}
%\published{xx.}
\maketitle
\begin{abstract}
  The energy of a graph $G$ is equal to the sum of the absolute values of the eigenvalues of $G$ ,  which in turn is equal to the sum of the singular values of the adjacency matrix of $G$.  Let $X$,  $Y$ and $Z$ be matrices,  such that $X+Y= Z$.  The Ky Fan theorem establishes an  inequality between the sum of the singular values of $Z$ and the sum of the sum of the singular values of $X$ and $Y$.  This theorem is applied in the theory of graph energy,  resulting in several new inequalities, as well as new proofs of some earlier known inequalities.\\[1mm]
\end{abstract}
%\setstretch{1.30}
\section{Introduction}
In this paper, we are concerned with simple graphs. Let $G = (V,E)$ be a simple graph, with nonempty vertex set $V= \{v_1,\ldots,v_n\}$ and edge set $E= \{e_1,\ldots,e_m\}$. That is to say, $G$ is a simple $(n,m)$-graph.
Let $\omega$ be a vertex weight of $G$, i.e., $\omega$ is a function from the set of vertices of $G$ to the set of positive real numbers. $G$ is called $\omega$-regular if for any $u,v\in V(G)$, $\omega(u)=\omega(v)$.
Observe that a well-know vertex weight of a graph is the vertex degree weight assigning to each vertex its degree. Let us denote it by $deg$.

The diagonal matrix of order $n$ whose $(i,i)$-entry is $\omega(v_i)$ is called the diagonal vertex weight matrix of $G$ with respect to  $\omega$   and is denoted by $D_\omega(G)$, i.e., $D_\omega(G)=\mathrm{diag}(\omega(v_i),\ldots,\omega(v_n))$ . The adjacency matrix $A(G)= (a_{ij})$ of $G$ is a $(0,1)$-matrix defined by $a_{ij}= 1$ if and only if the vertices $v_i$ and $v_j$ are adjacent. Then the matrices $L_{deg}(G)=D_{deg}(G)-A(G)$
and $L^\dag_{deg}(G)=A(G)+D_{deg}(G)$ are called Laplacian  and  signless Laplacian  matrix of $G$, respectively (see \cite{Gro-Mer1994}, \cite{Gro-Mer-sun1990}, \cite{Mer1995}, \cite{Mer1994}, \cite{Mohar1991} and \cite{Mohar2004}). Let us generalize these matrices for arbitrary vertex weighted graphs. Let $G$ be a simple graph with the vertex weight  $\omega$. Then we shall call the matrices $L_\omega(G)=D_\omega(G)-A(G)$
and $L^\dag _\omega(G)=A(G)+D_\omega(G)$  the weighted  Laplacian  and the weighted signless   Laplacian matrix of $G$ with respect to the vertex weight $\omega$.

Let  $X = \{x_1 ,x_2,...,x_n\}$ be a data set of real numbers. The \emph{mean absolute
deviation} (often called the mean deviation) $\mathrm{MD}(X)$ and variance $\mathrm{Var}(X)$ of $X$ is defined as
\begin{linenomath*}
\begin{equation*}
\mathrm{MD}(X)= \dfrac{1}{n}\sum_{i=1}^n|x_i-\overline{x}|,\quad\quad \mathrm{Var}(X)= \dfrac{1}{n}\sum_{i=1}^n(x_i-\overline{x})^2
\end{equation*}
\end{linenomath*}
where $\overline{x}=\dfrac{\sum_{i=1}^{n}x_i}{n}$ is the arithmetic mean of the distribution.
Note that an easy application of the Cauchy-Schwarz inequality gives that the mean deviation is a
lower bound on the standard deviation (see \cite{CaversThesis}).
\begin{linenomath*} \begin{equation}\label{MDVAR}
  \mathrm{MD}(X)\leq \sqrt{\mathrm{Var}(X)}.
\end{equation}\end{linenomath*}
The mean deviation and variance of $G$ with respect to $\omega$, denoted by $\mathrm{MD}_\omega(G)$ and  $\mathrm{Var}_\omega(G)$, respectively, is defined as
\begin{linenomath*} \begin{equation*}
\mathrm{MD}_\omega(G)=\mathrm{MD}(\omega(v_1),\ldots,\omega(v_n)),\quad \quad \mathrm{Var}_\omega(G)=\mathrm{Var}(\omega(v_1),\ldots,\omega(v_n)).
\end{equation*} \end{linenomath*}
It follows from  Eq. \eqref{MDVAR} that $\mathrm{MD}_\omega(G)\leq \sqrt{\mathrm{Var}_\omega(G)}$.
It is worth mentioning that $\mathrm{Var}_{\deg}(G)$ is well-investigated graph invariant (see \cite{Bell} and \cite{Gut-Paule2002}).
Let $\lambda_1, \lambda_2,\ldots , \lambda_n$ be eigenvalues of the adjacency matrix $A(G)$ of graph $G$. It is  known that $\sum_{i=1}^{n}\lambda_i=0$. The notion of the energy $E(G)$ of an $(n,m)$-graph $G$ was introduced by Gutman in connection with the $\pi$-molecular energy (see \cite{I.Gutman}, \cite{Gutman}, \cite{Gutm-zar} and \cite{Ind-Vijay}). It is defined as
\begin{linenomath*} \begin{equation*}E (G)=\sum_{i=1}^n \vert\lambda_i\vert=n \mathrm{MD}(\lambda_1, \lambda_2,\ldots , \lambda_n) .\end{equation*} \end{linenomath*}
For details of the theory of graph energy see \cite{Gutman}, \cite{Gut-Pol} and \cite{LiShiGut}.

Let $M\in\mathbb{C}^{n\times n}$ be Hermitian with singular values $s_i(M), i=1,2,\ldots,n$. If $\lambda_i(M), i=1,2,\ldots,n$ are eigenvalues of $M$, then $s_i(M)=|\lambda_i(M)|$, $i = 1, 2,\ldots, n$. Getting motivated from this fact,  Nikiforov established the concept of matrix energy by analogy with graph energy \cite{V.Nik}.
Let $M\in\mathbb{C}^{n\times n}$ with singular values $s_i(M), i=1,2,\ldots,n$. Then the energy of $M$, denoted by
$E(M)$, is defined as $s_1(M)+s_2(M)+\ldots +s_n(M)$. Consequently, if $M\in\mathbb{C}^{n\times n}$ is Hermitian with eigenvalues $\lambda_1(M),\lambda_2(M),\ldots,\lambda_n(M)$, we have
\begin{linenomath*} \begin{equation*}
E(M)=\sum_{i=1}^n|\lambda_i(M)|.
\end{equation*} \end{linenomath*}

Let $n\ge\mu_1, \mu_2,\ldots , \mu_n=0$ be eigenvalues of Laplacian matrix $L(G)$ of graph $G$.
It is  known that $\sum_{i=1}^{n}\mu_i=2m$.
Gutman and Zhou defined the Laplacian energy of an $(n,m)$-graph $G$ for the first time (see \cite{Gut-Zhou-Lap} ) as
\begin{linenomath*} \begin{equation*} LE(G)=\sum_{i=1}^n\Big\vert\mu_i-\frac{2m}{n}\Big\vert=n\mathrm{MD}(\mu_1,\ldots, \mu_n).\end{equation*} \end{linenomath*}
Numerous results on the Laplacian energy have been obtained, see for instance
\cite{T.Ale}, \cite{Dasa.Moj.Gut}, \cite{N.N.M}, \cite{I.N.M}, \cite{Romatch2009}, \cite{KeyFanLAA2010} and \cite{B.I}.
Note that in the definition of Laplacian energy $\dfrac{2m}{n}$ is the average vertex degree of $G$. This motivates us to extend their definition to the graphs equipped with an arbitrary vertex weight. Let $G$ be
 a graph with the vertex set $V= \{v_1,\ldots,v_n\}$ and with an arbitrary vertex weight $\omega$.
Let $\mu_1, \mu_2,\ldots , \mu_n$ be eigenvalues of the weighted Laplacian matrix $L_\omega(G)$ of graph $G$ with respect to the vertex weight $\omega$. Then we propose the weighted  Laplacian energy $LE_\omega(G)$ of $G$ with respect to the vertex weight $\omega$ as
\begin{linenomath*} \begin{equation}\label{eqn:wLE}
 LE_\omega (G)=\sum_{i=1}^n\big|\mu_i - \overline{\omega}\big|=n\mathrm{MD}(\mu_1,\ldots, \mu_n),
\end{equation}\end{linenomath*}
where
\begin{linenomath*} \begin{equation*}\overline{\omega}=\dfrac{\sum_{i=1}^{n}\omega(v_i)}{n} \quad \mbox{and} \quad \sum_{i=1}^{n}\mu_i=n\overline{\omega}.\end{equation*} \end{linenomath*}
Note that $LE_{deg}(G)=LE(G)$.

Let $I_s$ be the unit matrix of order $s$. For the considerations that follow it will be necessary to note
that instead via  Eq. \eqref{eqn:wLE}, the  weighted Laplacian energy can be expressed also as
\begin{linenomath*} \begin{equation}\label{eqn:Lapp}
LE_\omega(G)=E(L_\omega(G)-\overline{\omega}I_n).
\end{equation}
\end{linenomath*}
The following results are already known.

The next lemma is known for the vertex degree weight \cite{Cvet}; Its proof for an arbitrary vertex weight is done in a similar fashion.

\begin{lemma}\label{lem:Bi}

Let $G$ be a bipartite graphs with n vertices and with a vertex weight $\omega$.
Then $L_\omega(G)$ and $L_\omega^\dag(G)$ are similar.
\end{lemma}
%\begin{proof}
%weighted Laplacian and weighted signless  Laplacian matrices with A short proof is included for the sack of completeness.
%Let $G$ be a bipartite graphs with two parts of size $r$ and $s$ ($n=r+s$). By an appropriate ordering of vertices, the adjacency  matrix of $G$ is of the form
%$
%A(G)=
%\begin{pmatrix}
%O & B\\
%B^T & O
%\end{pmatrix}
%$
%, where $B$ is $r \times s$. Let us define an $n \times n$ idempotent matrix
%$
%S=
%\begin{pmatrix}
%-I_r & O\\
%O &  I_s
%\end{pmatrix}
%$. It is easy to check that $D_\omega=S^{-1}D_\omega S$ and $A=S^{-1}(-A)S$.
%%conclufion $A$ and $-A$ are similar to and the similariy shown with the symbol $\sim$. we write  $A\sim(-A)$,
%It follows that
%\begin{linenomath*} \begin{equation*}
%D_\omega+A=S^{-1}D_\omega S+S^{-1}(-A)S= S^{-1}(D_\omega S-AS)=S^{-1}(D_\omega-A)S.
%\end{equation*} \end{linenomath*}
%Then the weighted Laplacian matrix and the weighted signless   Laplacian matrix are similar.
%\end{proof}

\begin{lemma}\rm{\cite[Section 7.1, Ex. 2]{Horn-Joh}}\label{le:Key}
If $A =(a_{ij})_{i,j=1}^{n}$ is a positive semi-definite matrix and $a_{ii}=0$ for some $i$, then $a_{ij} =0=a_{ji}$, $j=1,\ldots,n$.
\end{lemma}

Theorem \ref{thm:KyFan}, supporting the concept of matrix energy proposed by Nikiforov, was first obtained by Ky Fan \cite{k.Fan} using a variational principle. It also appears in Gohberg and Krein \cite{Gohb-Ker} and in Horn and Johnson \cite{Horn-Joh}. No equality case is discussed in these references. Thompson \cite{R.C,R.C.Thom} employs polar decomposition theorem and
inequalities due to Fan and Hoffman \cite{K.A.J} to  obtain its equality case. Day and So \cite{Day-So-SVI-GEC} give the details of a proof for the inequality and the case of equality.
\begin{theorem}\label{thm:KyFan}
Let $A$ and $B$ be two complex square matrices of size $n$ ($A,B\in \mathbb{C}^{n\times n}$) and let $C = A + B$. Then
\begin{linenomath*} \begin{equation}\label{eqn:KeyFan}
E(C)\leq E (A)+E (B)  .
\end{equation}\end{linenomath*}
Moreover equality holds if and only if there exists an unitary matrix $P$ such that $PA$ and $PB$ are both positive semi-definite matrices  .
\end{theorem}

Let $A$ be a complex matrix of size $n$ ($A\in \mathbb{C}^{n\times n}$). Let us denote the Hermitian adjoint of $A$ by $A^\ast$.  Then both $A^\ast A$ and $AA^\ast$ are Hermitian positive semi-definite matrices with the same nonzero eigenvalues. In particular $A^\ast A$ and $AA^\ast$ are diagonalizable with real non-negative eigenvalues. Then by spectral theorem for complex matrices we may define $\abs{A}:=({A^\ast A})^{1/2}$.
Here we present the following version of the polar decomposition theorem \cite{Horn-Joh}.

\begin{theorem}\label{thm:PDT}
For $A\in \mathbb{C}^{n\times n}$, there exist positive semi-definite matrices $X,Y \in \mathbb{C}^{n\times n}$ and unitary matrices $P,F \in \mathbb{C}^{n\times n}$ such that $A =PX = YF$. Moreover, the matrices $X,Y$ are unique,
$X = |A|$, $Y =|A^\ast|$. The matrices $P$ and $F$ are uniquely determined if and only if $A$ is non-singular.
\end{theorem}

There is a great deal of analogy between the properties of $E(G)$ and $LE_{\omega}(G)$, but also some significant
differences. These similarities and dissimilarities has been investigated \cite{Sh-Pa-weiLa}. In this paper we apply Theorem \ref{thm:KyFan} in the theory of graph energy,  resulting in several new inequalities, as well as new proofs of some earlier known inequalities. It is worth mentioning that the idea of this paper inspired from \cite{Romatch2009} and \cite{KeyFanLAA2010}; Our proofs are based on those of these references.

\section{Graphs $G$ for which $LE_\omega(G)=E(G)+E(D_\omega(G)-\overline{\omega}I_n)$}

In the case of vertex degree weight, the inequality in the following theorem was proved in \cite{KeyFanLAA2010}, whereas the equality in  Eq. \eqref{eqn:Lap} was investigated in \cite{Romatch2009}. Based on their proof, we generalize their results for a connected graph with an arbitrary vertex weight.

\begin{theorem}\label{thm:MD}
Let $G$ be a connected graph with $n$ vertices and with a vertex weight $\omega$. Then
\begin{linenomath*} \begin{equation}\label{eqn:Lap}
 LE_\omega(G)\leq n\mathrm{MD}_\omega(G)+E(G).
\end{equation}\end{linenomath*}
Moreover the equality in  Eq. \eqref{eqn:Lap} holds if and if $G$ is  $\omega$-regular.
\end{theorem}
\begin{proof}
We Know that
\begin{linenomath*} \begin{equation}\label{eqn:LomegaFan}
L_\omega(G)-\overline{\omega}I_n= (D_\omega(G)-\overline{\omega}I_n)+(-A(G)).
\end{equation}\end{linenomath*}
Note that $D_\omega(G)-\overline{\omega}I_n$ is a diagonal matrix whose eigenvalues are $\omega(v_i)-\overline{\omega}$, $i=1,\ldots,n$.
It follows from Theorem \ref{thm:KyFan} that
\begin{linenomath*} \begin{equation*}
\sum_{i=1}^ns_i(L_\omega(G)-\overline{\omega}I_n)\leq\sum_{i=1}^ns_i(D_\omega(G)-\overline{\omega}I_n)+
\sum_{i=1}^ns_i(-A(G)).
\end{equation*} \end{linenomath*}
Therefor
\begin{linenomath*} \begin{equation*}
 LE_\omega(G)\leq \sum_{i=1}^n|\omega(v_i)-\overline{\omega}|+\sum_{i=1}^n|\lambda_i(-A(G))|.
 \end{equation*} \end{linenomath*}
 Then, due to the similarity between $A(G)$ and $-A(G)$, we have
\begin{linenomath*} \begin{equation*} LE_\omega(G)\leq n\mathrm{MD}_\omega(G)+E(G).\end{equation*} \end{linenomath*}

Let $G$ be a  $\omega$-regular graph with eigenvalues $\lambda_1,\ldots,\lambda_n$. Then  $\overline{\omega}=\omega(v_i)$ for each $1\leq i\leq n$ and $L_\omega(G)=\overline{\omega}I_n-A(G)$.
It follows that $\overline{\omega}-\lambda_1,\ldots,\overline{\omega}-\lambda_n$ are all the eigenvalues of
$L_\omega(G)$. Therefore, by  Eq. \eqref{eqn:wLE} we have
\begin{linenomath*} \begin{equation*} LE_\omega(G)=E(G).\end{equation*} \end{linenomath*}

 Conversely, suppose that the equality in  Eq. \eqref{eqn:Lap} holds. Without loss of generality, we may assume that
 $\omega(v_1)=\max\{\omega(v_i) \mid 1\leq i\leq n)\}$. Suppose on the contrary that $G$ is not $\omega$-regular. Therefore
\begin{linenomath*} \begin{equation}\label{eqn:Del}
\omega(v_1)>\overline{\omega}.
\end{equation}\end{linenomath*}
Let $ a_i:= \omega(v_i) - \overline{\omega} $ for $i =1,\ldots,n$. We have $a_1 > 0$, via  Eq. \eqref{eqn:Del}.
Due to the equality in  Eq. \eqref{eqn:Lap}, we may apply Theorem \ref{thm:KyFan} to
 Eq. \eqref{eqn:LomegaFan}. Therefore, there exists a unitary matrix $P$ such that $X=P(D_\omega(G)-\overline{\omega}I_n)$ and $Y= P(-A(G))$ are both positive  semi-definite. Hence $P^\ast X$ and $P^\ast Y$ are polar decompositions of the matrices $D_\omega(G)- \overline{\omega}I_n$ and $-A(G)$, respectively. It follows from Theorem \ref{thm:PDT} that $X= |D_\omega(G)- \overline{\omega}I_n|$ and $Y = |A(G)|$. Therefore
 $X =\mathrm{diag}(|a_1|,|a_2|,\ldots,|a_n|)$ . Setting
\begin{linenomath*} \begin{equation*}
P^\ast=
\begin{pmatrix}
q_{11} & \cdots & q_{1n}\\
\vdots   & \ddots &\vdots\\
q_{n1} & \cdots & q_{nn}
\end{pmatrix}
,\qquad
A(G)=
\begin{pmatrix}
0 & a_{12} &  \cdots & a_{1n}\\
a_{12} & 0 & \cdots & a_{2n}\\
\vdots & \ddots & \ddots & \vdots\\
a_{1n} & a_{2n} & \cdots & 0
\end{pmatrix}
,\end{equation*} \end{linenomath*}
\quad

$P^\ast X = D_\omega(G) - \overline{\omega} I_n$, implies

\begin{linenomath*} \begin{equation*}
\begin{pmatrix}
q_{11} & \cdots & q_{1n}\\
\vdots   & \ddots &\vdots\\
q_{n1} & \cdots & q_{nn}
\end{pmatrix}
\begin{pmatrix}
|a_1| & ~ & ~ & ~\\
~ & \ddots & ~ & ~\\
~ & ~ &~ &|a_n|
\end{pmatrix}
=
\begin{pmatrix}
a_1 & ~ & ~ & ~\\
~ & \ddots & ~ & ~\\
~ & ~ &~ & a_n
\end{pmatrix}
.\end{equation*} \end{linenomath*}
Then,
\begin{linenomath*} \begin{equation*}
\begin{pmatrix}
|a_1|q_{11} & |a_2|q_{12} & \cdots & |a_n|q_{1n}\\
|a_1|q_{21} & |a_2|q_{22} & \cdots & |a_n|q_{2n}\\
\vdots   & \vdots & \ddots &\vdots\\
|a_1|q_{n1} & |a_2|q_{n2} & \cdots & |a_1|q_{nn}
\end{pmatrix}
=
\begin{pmatrix}
a_1 & ~ & ~ & ~ & ~\\
~ & a_2 & ~ & ~ & ~ \\
~ & ~& \ddots & ~ & ~\\
~ & ~ &~ & ~ & a_n
\end{pmatrix}
.\end{equation*} \end{linenomath*}
Equality at first column imposes $q_{11}= 1$ and $q_{i1} =0,{i=2,\ldots,n}$. It follows that

\begin{linenomath*} \begin{equation*}P=
\begin{pmatrix}
1 & 0 & \cdots & 0\\
q_{12} & ~ & \cdots & q_{1n}\\
\vdots   & ~ & \ddots &\vdots\\
q_{1n} & ~ & \cdots & q_{nn}
\end{pmatrix}.
\end{equation*} \end{linenomath*}
We must then have
\begin{linenomath*} \begin{equation*}
Y =-
\begin{pmatrix}
1 & 0   & \cdots & 0\\
q_{12} & ~ & \cdots & q_{1n}\\
\vdots   & ~ & \ddots &\vdots\\
q_{1n} & ~ & \cdots & q_{nn}
\end{pmatrix}
\begin{pmatrix}
0 & a_{12} &  \cdots & a_{1n}\\
a_{12} & 0 & \cdots & a_{2n}\\
\vdots & \ddots & \ddots & \vdots\\
a_{1n} & a_{2n} & \cdots & 0
\end{pmatrix}
=-
\begin{pmatrix}
0 & a_{12} &  \cdots & a_{1n}\\
\ast & 0 & \cdots & \ast \\
\vdots & \ddots & \ddots & \vdots\\
\ast & a_{2n} & \cdots & \ast
\end{pmatrix}.
\end{equation*} \end{linenomath*}
The previous matrix is positive semi-definite and by Lemma \ref{le:Key}, we obtain $a_{1j} =0$, $j=2,\ldots,n$. This contradicts our assumption that $G$ is a connected graph and the result follows. $\hfill{\blacksquare}$
\end{proof}

\section{Graphs $G$ for which $LE_\omega(G)=E(G)$}

In Theorem \ref{thm:MD} we showed that if $G$ is a  $\omega$-regular graph, then
\begin{linenomath*} \begin{equation}\label{eqn:Reg}
LE_\omega(G)=E(G).
\end{equation}\end{linenomath*}
In what follows we consider the converse argument.

In the case of vertex degree weight, the first part of the following theorem was proved in \cite{KeyFanLAA2010}, whereas the second part was proved in \cite{Romatch2009}. Based on their proof, we generalize their results for a connected graph with an arbitrary vertex weight.
\begin{theorem}\label{thm:bi-Reg}
Let $G$ be a bipartite graph with a vertex weight $\omega$. Then
\begin{linenomath*} \begin{equation}\label{eqn:Reg-bip}
LE_\omega(G)\ge E(G).
\end{equation}\end{linenomath*}
Moreover, the equality in  Eq. \eqref{eqn:Reg-bip} holds if and only if $G$ is a  $\omega$-regular graph.
\end{theorem}
\begin{proof}
From the definition of weighted  Laplacian matrix and weighted signless   Laplacian matrix, it is clear that
\begin{linenomath*} \begin{equation}\label{eqn:Regg}
\big(L^\dag _\omega(G)- \overline{\omega}I_n\big)-\big(L_\omega(G)-\overline{\omega}I_n\big)=2A(G).
\end{equation}\end{linenomath*}

If $G$ is bipartite, then it follows from Lemma \ref{lem:Bi} that $L_\omega(G)$ and $L^\dag_\omega(G)$ have the same spectra and therefore
\begin{linenomath*} \begin{equation*}
\sum_{i=1}^ns_i(L_\omega^\dag(G)-\overline{\omega}I_n)=\sum_{i=1}^ns_i(L_\omega-\overline{\omega}I_n)=\sum_{i=1}^ns_i(-[L_\omega(G)-\overline{\omega}I_n])=
LE_\omega(G).
\end{equation*} \end{linenomath*}
So by  Theorem \ref{thm:KyFan}, $LE_\omega(G)\ge E(G)$.

Let $G$ be a  $\omega$-regular graph. Then by Theorem \ref{thm:MD},
the equality in  Eq. \eqref{eqn:Reg-bip} holds. Conversely, suppose that the equality in  Eq. \eqref{eqn:Reg-bip} holds.
Therefore,
\begin{linenomath*} \begin{equation*}
 E\Big((L^\dag_\omega(G)- \overline{\omega}I_n)-(L_\omega(G)-\overline{\omega}I_n)\Big)=2 E(G) = E(G)+E (G)= LE_\omega(G)+LE_\omega(G) .
\end{equation*} \end{linenomath*}
%Bearing in mind that $G$ is bipartite we obtain
%and the fact that $2A(G)=L^\dag_\omega(G)-L_\omega(G)$
Since $G$ is bipartite it follows from Lemma \ref{lem:Bi} that
\begin{linenomath*} \begin{equation}\label{eqn:Re}
 E\Big((L^\dag_\omega(G)-\overline{\omega}I_n)-(L_\omega(G)-\overline{\omega}I_n)\Big)=E\big(L_\omega^\dag(G)-\overline{\omega}I_n\big)+
 E\Big(-\big(L_\omega(G)-\overline{\omega}I_n\big)\Big).
\end{equation}\end{linenomath*}
 Therefore, Theorem \ref{thm:KyFan} asserts that there exists a unitary matrix $P$ , such that
 \begin{linenomath*} \begin{equation}\label{eqn:Ree}
X = P\Big(L^\dag _\omega(G)- \overline{\omega}I_n\Big)  ~~\mbox{and} ~~ Y = P\Big(-\big(L_\omega(G)-\overline{\omega}I_n\big)\Big),
\end{equation}\end{linenomath*}
are both positive semi-definite matrices. Hence $P^\ast X$ and $P^\ast Y$ are polar decompositions of
\begin{linenomath*} \begin{equation*}
L^\dag_\omega(G)- \overline{\omega}I_n ~~ \text{and} ~~ -\big(L_\omega(G)- \overline{\omega}I_n\big),
\end{equation*} \end{linenomath*}
respectively. By Theorem \ref{thm:KyFan} we obtain
\begin{linenomath*} \begin{equation*}
X = |L^\dag_\omega(G)- \overline{\omega}I_n| ~~\text{and}~~ Y = |-\big(L_\omega(G)-\overline{\omega}I_n\big)| .
\end{equation*} \end{linenomath*}
In view of the fact that $G$ is bipartite, we conclude that $X = Y$. Therefore, it follows from  Eq. \eqref{eqn:Ree} that
\begin{linenomath*} \begin{equation*}
L^\dag_\omega(G)+L_\omega(G)=2\overline{\omega}I_n,
\end{equation*} \end{linenomath*}
implying the result. $\hfill{\blacksquare}$
\end{proof}

In the case of vertex degree weight, the  next theorem was proved in \cite{KeyFanLAA2010} and  based on their proof, we get also the following theorem.
%generalize their results for a connected graph with an arbitrary
%vertex weight.
%In an analogous manner we get also the following
\begin{theorem}\label{th:}
Let $G$ be a bipartite graph  with $n$ vertices and  with a vertex weight $\omega$. Then
\begin{linenomath*} \begin{equation}\label{eqn:}
\max\Big\{n\mathrm{MD}_\omega(G), E(G)\Big\}\leq LE_\omega(G)\leq n\mathrm{MD}_\omega(G)+E(G).
\end{equation}\end{linenomath*}
\end{theorem}
\begin{proof}
The right side inequality is a direct consequent of Theorem \ref{thm:MD}. Let us prove the left one.
It is easy to see that
  \begin{linenomath*} \begin{equation*}
L^\dag_\omega(G)+L_\omega(G)=2D_\omega(G),%\overline{\omega}I_n
\end{equation*} \end{linenomath*}
from which

\begin{linenomath*} \begin{equation*}\Bigg((L^\dag_\omega(G)- \overline{\omega}I_n\Bigg)+\Bigg(L_\omega(G)-\overline{\omega}I_n)\Bigg)=2\Bigg(D_\omega(G)-\overline{\omega}I_n\Bigg).\end{equation*} \end{linenomath*}
It follows from Theorem \ref{thm:KyFan} that

\begin{linenomath*} \begin{equation*}E\Bigg((L^\dag_\omega(G)-\overline{\omega}I_n\Bigg)+E\Bigg(L_\omega(G)-\overline{\omega}I_n)\Bigg)\ge
2E\Bigg(D_\omega(G)-\overline{\omega}I_n\Bigg)=2n\mathrm{MD}_\omega(G).\end{equation*} \end{linenomath*}
In the other hand, since $G$ is bipartite, it follows from Lemma \ref{lem:Bi} that
\begin{linenomath*} \begin{equation*}LE_\omega(G)=E\Bigg((L^\dag_\omega(G)-\overline{\omega}I_n\Bigg)=E\Bigg(L_\omega(G)-\overline{\omega}I_n)\Bigg).\end{equation*} \end{linenomath*}
Therefore
\begin{linenomath*} \begin{equation}\label{MD-LE}
  LE_\omega(G)\ge n\mathrm{MD}_\omega(G).
\end{equation}\end{linenomath*}
Hence, the result follow from  Eq. \eqref{MD-LE} and  Theorem \ref{thm:bi-Reg}. $\hfill{\blacksquare}$
\end{proof}

\section{An upper bound on the Laplacian matrix energy for the disjoint union of graphs}

Here and throughout this section, $\bigoplus$ denotes the block matrix direct sum \cite{Horn-Joh}.

 Let $k\in\mathbb{N}$. Suppose that for each  $1\le i \leq k$, $G_i=(V_i,E_i)$ is an $(n_i,m_i)$-graph with the vertex set $V_i$  and the edge set $E_i$. Let $V_i$'s are mutually disjoint. In this case the \emph{disjoint union} of $G_i$'s, denoted by $\bigcup_{i=1}^{k}G_i$, is a non-connected graph with the vertex set $\bigcup_{i=1}^{k}V_i$ and the
 edge set $\bigcup_{i=1}^{k}E_i$. It is easy to see that
 \begin{linenomath*} \begin{equation}\label{eqn}
   A(\bigcup_{i=1}^{k}G_i)=\bigoplus_{i=1}^k A(G_i).
 \end{equation}\end{linenomath*}
Moreover, if $\omega_i$  is a vertex weight, assigned to $G_i$, then $\bigcup_{i=1}^{k}G_i$ inherits naturally a vertex degree weight from its components. This weight is nothing but $\omega:=\bigcup_{i=1}^{k}\omega_i$, i.e., For each $v\in \bigcup_{i=1}^{k}V_i$, $\omega(v)=\omega_i(v)$ if and only if $v\in V_i$.
%The following equality follows immediately from the statement,
Note that $\overline{\omega}$ is a convex combination of $\overline{\omega}_i$, $i=1,\ldots,k$, since
 \begin{linenomath*} \begin{equation}\label{eqn:L15}
  \overline{\omega}=\Big(\dfrac{1}{\sum_{j=1}^kn_j}\Big)\Big(\sum_{i=1}^k \sum_{v\in V_i}\omega_i(v)  \Big)= \sum_{i=1}^k\Big(\dfrac{n_i}{\sum_{j=1}^kn_j}\Big)\overline{\omega}_i.
  \end{equation}\end{linenomath*}
Moreover
\begin{linenomath*}
\begin{equation*}
\overline{\omega} \ge \overline{\omega_i}, \quad i=1,\ldots,k.
\end{equation*}
\end{linenomath*}
In the case of vertex degree weight, the  next theorem was proved in \cite{Romatch2009} and  based on their proof, we get also the following result.
\begin{theorem}\label{the:th8}
 Let $k\in\mathbb{N}$. Suppose that for each  $1\le i \leq k$, $G_i$ is a graph
 with $n_i$ vertices and with a vertex weight $\omega_i$. Then
 \begin{linenomath*} \begin{equation}\label{eqn:L14}
  LE_\omega(\bigcup_{i=1}^{k}G_i) \leq \sum_{i=1}^kLE_{\omega_i}(G_i)+\sum_{i=1}^k\Big|\overline{\omega}_i-\overline{\omega}\Big|n_i.
\end{equation}\end{linenomath*}
 Equality holds if and only if $\overline{\omega}_i=\overline{\omega}$ for all $i =1,\ldots,k$.
 \end{theorem}

\begin{proof}
 In order to simplify the writing and omit some subscripts,  for each  $1\le i \leq k$, we denote $I_{n_i}$ and $\overline{\omega}_i-\overline{\omega}$
 by $I_i$ and $b_i$, respectively. It is clear that
\begin{linenomath*} \begin{equation}\label{eqn:bsummat}
  L_\omega(G)-\overline{\omega}I_n= \bigoplus_{i=1}^k\big(L_{\omega_i}(G_i)-\overline{\omega}I_i\big)= \bigoplus_{i=1}^k\big(L_{\omega_i}(G_i)-\overline{\omega}_iI_i\big)+\bigoplus_{i=1}^kb_iI_i
\end{equation}\end{linenomath*}
Therefore, as a consequence of  Eq. \eqref{eqn:Lapp} and Theorem \ref{thm:KyFan}, the inequality in  Eq. \eqref{eqn:L14} follows.

%On the equality case, the condition is sufficient
Now let us consider the the equality case in  Eq. \eqref{eqn:L14}.
Let $\overline{\omega}_i=\overline{\omega}$ for all $i =1,\ldots,k$. Therefore the matrix $\displaystyle\bigoplus_{i=1}^kb_iI_i$ is zero and consequently it follows from  Eq. \eqref{eqn:bsummat} that the equality in  Eq. \eqref{eqn:L14} holds.

Conversely suppose on the contrary that
%the equality in  Eq. \eqref{eqn:L14} is true and suppose that, the equalities
%$\overline{\omega}-i =\overline{\omega}$ for all $i =1,\ldots,k$ fail. Therefore, by  Eq. \eqref{eqn:L15}
there exists $1\leq l\leq k$ such that
$\overline{\omega}_l > \overline{\omega}$. We may assume that $l= 1$. As a consequence of Theorem \ref{thm:KyFan},  Eq. \eqref{eqn:bsummat} and the equality in  Eq. \eqref{eqn:L14}, there exists a unitary matrix $P$ such that
\begin{linenomath*} \begin{equation*}
X=P\bigoplus_{i=1}^k\big(L_{\omega_i}(G_i)-\overline{\omega}_iI_i\big)
~~~\mbox{and}~~~
Y= P\bigoplus_{i=1}^kb_iI_i,
\end{equation*} \end{linenomath*}
are both positive semi-definite. Hence $P^\ast X$ and $P^\ast Y$ are polar decompositions of the matrices
\begin{linenomath*} \begin{equation*}
\bigoplus_{i=1}^k\big(L_{\omega_i}(G_i)-\overline{\omega}_iI_i\big)
~~~\mbox{and}~~~
\bigoplus_{i=1}^kb_iI_i,
\end{equation*} \end{linenomath*}
respectively. By Theorem \ref{thm:PDT}, we arrive at
\begin{linenomath*} \begin{equation}\label{eqn:L16}
Y=\bigoplus_{i=1}^k|b_i|I_i= P\bigoplus_{i=1}^kb_iI_i
\end{equation}\end{linenomath*}
We can write the unitary matrix $P$ as
\begin{linenomath*} \begin{equation}\label{eqn:L17}
P=
\begin{pmatrix}
P_{11} & P_{12} & \cdots & P_{1k}\\
P_{21} & P_{22} & \cdots & P_{2k}\\
\vdots   & ~         & \ddots & \vdots \\
P_{k1} & \cdots    & ~       & P_{kk}
\end{pmatrix},
\end{equation}\end{linenomath*}
with the diagonal matrices $P_{jj}$,$j=1,\ldots,k$ of order $n_j$, respectively. From  Eq. \eqref{eqn:L16} we have
\begin{linenomath*} \begin{equation*}
\begin{pmatrix}
|b_1|I_1 & 0 & \cdots & 0\\
0 & |b_2|I_2 & \cdots & 0\\
\vdots   & ~         & \ddots & \vdots \\
0 & \cdots    & 0       & |b_k|I_k
\end{pmatrix}=
\begin{pmatrix}
P_{11} & P_{12} & \cdots & P_{1k}\\
P_{21} & P_{22} & \cdots & P_{2k}\\
\vdots   & ~         & \ddots & \vdots \\
P_{k1} & \cdots    & ~       & P_{kk}
\end{pmatrix}
\begin{pmatrix}
b_1I_1 & 0 & \cdots & 0\\
0 & b_2I_2 & \cdots & 0\\
\vdots   & ~         & \ddots & \vdots \\
0 & \cdots    & 0      & b_kI_k
\end{pmatrix}
,\end{equation*} \end{linenomath*}
and then
\begin{linenomath*} \begin{equation}\label{eqn:L18}
\begin{pmatrix}
|b_1|I_1 & 0 & \cdots & 0\\
0 & |b_2|I_2 & \cdots & 0\\
\vdots   & ~         & \ddots & \vdots \\
0 & \cdots    & 0       & |b_k|I_k
\end{pmatrix}=
\begin{pmatrix}
b_1P_{11} & P_{12} & \cdots & P_{1k}\\
b_1P_{21} & P_{22} & \cdots & P_{2k}\\
\vdots   & ~         & \ddots & \vdots \\
b_1P_{k1} & \cdots    & ~       & P_{kk}
\end{pmatrix}.
\end{equation}\end{linenomath*}
As $b_1 =\overline{\omega}_1-\overline{\omega}>0$, via  Eq. \eqref{eqn:L18} we obtain $P_{11}= I_1$ and $P_{j1}= 0$, $j =2,\ldots,k$. Now it follows from $X=P\bigoplus_{i=1}^k\big(L_{\omega_i}(G_i)-\overline{\omega}_iI_i\big)$ that $L_{\omega_1}(G_1)-\overline{\omega}_1I_1$ is positive semi-definite. Now we have the required contradiction, since by the Rayleigh principle we find that
$L_{\omega_1}(G_1)-\overline{\omega}_1I_1$ has a negative eigenvalue. Hence the assertion follows.
\end{proof}

%\begin{definition}
%Definition will be written here
%\end{definition}
%
%\begin{equation}
%Equations will be written here
%\end{equation}
%
%\begin{example}
%Example will be written here
%\end{example}
%
%
%\begin{lemma}
%lemma will be written here
%\end{lemma}
%
%
%\begin{theorem}
%Theorem will be written here
%\end{theorem}
%
%\begin{proof}
%Proof will be written here
%\end{proof}
%
%
%\begin{proposition}
%Proposition will be written here
%\end{proposition}
%
%
%\section{Conclusion}

\end{document}